\documentclass[12pt]{article}
\usepackage{amsmath}
\usepackage{graphicx,psfrag,epsf}
\usepackage{enumerate}
\usepackage{natbib}
\usepackage{url} 
\usepackage{float}
\usepackage{bm}
\usepackage{booktabs,makecell}
\usepackage{amsthm,amssymb}
\usepackage{amsfonts}
\usepackage{dsfont}
\usepackage{mathtools}
\usepackage{authblk}

\usepackage{amsmath}

\DeclareMathOperator*{\argmin}{arg\,min}

\usepackage{bbold}
\usepackage{algorithm}
\usepackage{algorithmic}

\usepackage{amssymb}
\usepackage{hyperref}
\usepackage{multirow}
\usepackage{booktabs}
\usepackage{mathrsfs} 
\hypersetup{
     colorlinks = true,
     linkcolor=blue,
            urlcolor=blue,
            citecolor=blue
}
\usepackage{xr}
\newcommand{\blind}{0}

\begin{document}

\def\spacingset#1{\renewcommand{\baselinestretch}%
{#1}\small\normalsize} \spacingset{1}


\def\bmE{\boldsymbol{\mathcal{E}}}
\def\bmlambda{\boldsymbol{\lambda}}
\def\bmrho{\boldsymbol{\rho}}
\def\Cov{{\rm Cov}}
\def\rmvec{{\rm vec}}
\def\bmSigma{\boldsymbol{\Sigma}}
\def\mbR{\mathbb{R}}
\def\mbW{\mathbb{W}}
\def\bmI{\boldsymbol{I}}
\def\bmy{\boldsymbol{y}}
\def\bmvarepsilon{\boldsymbol{\varepsilon}}
\def\bmxi{\boldsymbol{\xi}}
\def\hbmbeta{\boldsymbol{\hat\beta}}
\def\tbmbeta{\boldsymbol{\tilde\beta}}
\def\bbmbeta{\boldsymbol{\bar\beta}}
\def\mL{\mathscr{L}}
\def\tmL{\widetilde {\mL}}
\def\lbk{\left\{}
\def\rbk{\right\}}
\def\lak{\left|}
\def\rak{\right|}
\def\lmk{\left[}
\def\rmk{\right]}
\def\lsk{\left(}
\def\rsk{\right)}
\def\bmx{\boldsymbol{x}}
\def\bmz{\boldsymbol{z}}
\def\bmv{\boldsymbol{v}}
\def\bmmu{\boldsymbol{\mu}}
\def\bmzeta{\boldsymbol{\zeta}}
\def\bmsigma{\boldsymbol{\sigma}}
\def\bmbeta{\boldsymbol{\beta}}
\def\var{{\rm var}}
\def\bml{\boldsymbol{1}}
\def\bmo{\boldsymbol{0}}
\def\diag{{\rm diag}}
\def\bmrho{\boldsymbol{\rho}}
\def\E{{\rm E}}
\def\bmgamma{\boldsymbol{\gamma}}
\def\bmalpha{\boldsymbol{\alpha}}
\def\bmvtheta{\boldsymbol{\vartheta}}
\def\laak{\left\|}
\def\raak{\right\|}
\def\OLS{{\rm OLS}}
\def\GLS{{\rm GLS}}
\def\FGLS{{\rm FGLS}}
\def\LS{{\rm LS}}
\def\FLS{{\rm FLS}}
\def\bmtheta{\boldsymbol{{\mathcal{B}}}}
\def\tr{{\rm tr}}
\def\mbY{\mathbb{Y}}
\def\mbX{\mathbb{X}}
\def\rmE{{\rm E}}
\def\rmP{{\rm P}}
\def\Cov{{\rm Cov}}
\def\Var{{\rm Var}}
\def\mS{{\mathcal{S}}}
\def\mU{{\mathcal{U}}}
\def\mN{{\mathcal{N}}}
\def\mV{{\mathcal{V}}}
\def\sgn{\operatorname{sgn}}
\def\bu{\boldsymbol u}
\def\mR{\mathbb{R}}
\def\bmY{\mathbf{Y}}
\def\bmW{\mathbf{W}}
\def\bmV{\mathbf{V}}
\def\bmQ{\mathbf{Q}}
\def\bmX{\mathbf{X}}
\def\bmSigma{\mathbf{\Sigma}}

\def\be{\begin{equation}}
\def\ee{\end{equation}} 
\def\ben{\begin{equation*}}
\def\een{\end{equation*}}
\def\bea{\begin{eqnarray}}
\def\eea{\end{eqnarray}}
\def\bda{\begin{eqnarray*}}
\def\eda{\end{eqnarray*}}
\numberwithin{equation}{section}

\newtheorem{definition}{Definition}
\newtheorem{theorem}{Theorem}
\newtheorem{proposition}{Proposition}
\newtheorem{corollary}{Corollary}
\newtheorem{assumption}{Assumption}
\newtheorem{lemma}{Lemma}
\newtheorem*{remark}{Remark}

\theoremstyle{definition}
\newtheorem{exmp}{Example}[section]
\AtEndDocument{\refstepcounter{theorem}\label{finalthm}}
\AtEndDocument{\refstepcounter{proposition}\label{finalprop}}
\newcommand{\pkg}[1]{{\normalfont\fontseries{b}\selectfont #1}} \let\proglang=\textsf \let\code=\texttt


\if0\blind
{
  \title{\bf Generalized Score Matching for Regression}
  \author[*]{Jiazhen Xu}
  \author[*]{Janice L. Scealy}
  \author[*]{Andrew T. A. Wood}
  \author[*]{Tao Zou}
  \affil[*]{Research School of Finance, Actuarial Studies and Statistics, Australian National University, Canberra ACT 2601, Australia}
  \maketitle
} \fi

\if1\blind
{
  \bigskip
  \bigskip
  \bigskip
  \begin{center}
    {\LARGE\bf Generalized Score Matching for Regression}
\end{center}
  \medskip
} \fi

\bigskip

\begin{abstract}
Many probabilistic models that have an intractable normalizing constant may be extended to contain covariates. Since the evaluation of the exact likelihood is difficult or even impossible for these models, score matching was proposed to avoid explicit computation of the normalizing constant. In the literature, score matching has so far only been developed for models in which the observations are independent and identically distributed (IID). However, the IID assumption does not hold in the traditional fixed design setting for regression-type models. To deal with the estimation of these covariate-dependent models, this paper presents a new score matching approach for independent but not necessarily identically distributed data under a general framework for both continuous and discrete responses, which includes a novel generalized score matching method for count response regression. We prove that our proposed score matching estimators are $\sqrt{n}$-consistent and asymptotically normal under mild regularity conditions. The theoretical results are supported by simulation studies and a real-data example. Additionally, our simulation results indicate that, compared to approximate maximum likelihood estimation, the generalized score matching produces estimates with substantially smaller biases in an application to doctoral publication data.
\end{abstract}
\noindent%
\vfill

\newpage
\spacingset{1.9} 

\section{Introduction}

Many probabilistic models contain an intractable normalization constant which does not have a closed form and is difficult to compute. Standard methods of estimation, such as maximum likelihood estimation, are computationally intensive for these models, if even feasible. Several methods of approximating intractable normalizing constants have been studied (\citealt{minka2003computing,brooks2011handbook,huber2015approximation}). However, approximating normalizing constants will introduce bias into estimation. To tackle this issue, score matching and its extensions have been developed (\citealt{hyvarinen2005estimation,hyvarinen2007some,vincent2011connection,lyu2012interpretation,song2020sliced}) to avoid the explicit computation of the normalizing constant. Score matching is a powerful method for performing parameter estimation in previously intractable models.

Score matching for continuous data is based on the Fisher divergence (\citealt{johnson2004information,dasgupta2008asymptotic}). In the score matching procedure, researchers from machine learning define the score function to be the derivative of the log density with respect to the sample data, i.e., a sample space derivative, as opposed to the derivative with respect to the parameter. In this paper, we use the sample space derivative definition of the score function. The principle of score matching is to minimize the expected squared distance between the score functions given by the parametrized density and the true density, respectively. We refer to this distance as the score matching objective function. \cite{hyvarinen2005estimation} showed that, under mild regularity conditions, the score matching objective function has a tractable version which does not depend on the unknown true probability distribution. However, we may not be able to obtain the tractable version of the objective function when the true density is not supported on all of real space. To deal with this problem, weight functions are introduced so that the boundary condition required for deriving a tractable objective function is satisfied (\citealt{hyvarinen2007some,yu2019generalized,liu2019estimating,scealy2020score}). 

Because the basic form of score matching is valid only for models which define a differentiable density function, two variants of score matching have been proposed for discrete data (\citealt{hyvarinen2007some,lyu2012interpretation}). \cite{hyvarinen2007some} extended score matching to deal with binary data and this extended method is called ratio matching. Ratio matching is based on minimizing the expected squared distance of the ratios of certain probabilities given by the model and the corresponding ratios in the observations. However, ratio matching and the generalized score matching proposed by \cite{lyu2012interpretation} cannot be applied to univariate discrete data, as these methods still suffer the explicit computation of the intractable normalization constant. 

In practice, score matching and its variants for probabilistic models has been widely used in graphical models (\citealt{yu2019generalized}), generative adversarial networks (\citealt{pang2020efficient}), data from Riemannian manifolds, such as the sphere data from Riemannian manifolds with boundary, especially compositional data (\citealt{mardia2016score,scealy2020score}). In many real-life situations, the structure of these probabilistic models may further depend on other extraneous factors in the form of explanatory variables or covariates, which result in covariate-dependent probabilistic models. Examples of such probabilistic models can be found in genetic studies (\citealt{yin2011sparse,cai2013covariate,cheng2014sparse}), computer vision (\citealt{gustafsson2020energy}) and network data analysis (\citealt{yuan2021community,zhao2021dimension}). This increasing demand for fitting covariate-dependent models with an intractable normalizing constant motivates the development of score matching approaches for regression-type models. \cite{hyvarinen2005estimation} presented a general framework for score matching which allows for the possibility of incorporating covariates but so far score matching for regression-type models under a fixed design has not been developed in the literature. Additionally, most of the variants of score matching for discrete data only focus on binary data. However, there is a need to develop score matching methods for regression-type models for count data as there exist some widely used probabilistic models with an intractable normalizing constant for count data, such as Conway-Maxwell-Poisson (CMP) regression models (\citealt{sellers2010flexible,sellers2020conway}).

This paper has the following novel contributions: (i) develop score matching for regression-type models with continuous response; (ii) propose a novel generalized score matching method for regression-type models for count data; and (iii) detailed development of the estimation theory of score matching for independent but not necessarily identically distributed (INID) observations under the fixed design setting for regression-type models. Under a general framework encompassing truncated regression models, size-biased sampling and CMP regression models, we show that the proposed score matching is a tractable estimation method which produces consistent and asymptotically normality distributed estimates of regression parameters. Additionally, our proposed generalized score matching for discrete data works for univariate and multidimensional variables and inherits the principle of the original score matching for continuous data which is to compare the slopes of the log densities.

The rest of this article is organized as follows. Section \ref{sect::model setting} introduces a unified framework for regression-type models with continuous response and discrete response. We present two particular cases of the generalized score matching method under the unified framework and discuss the theoretical properties of the estimators in Section \ref{sect::sm}. A doctoral publication example and Monte Carlo studies are presented in Section \ref{sect::case study} and Section \ref{sect::numerical study}, respectively, which indicate that our estimators perform well. The proofs of the theorems and propositions can be found in the Appendix and all the other technical details are relegated to the supplementary material for this paper. The supplementary material covers generalized score matching for discrete INID multivariate response and the detailed derivatives of our proposed score matching objective functions.

\section{Regression-Type Models with an Intractable Normalizing Constant}\label{sect::model setting}
As a starting point, we present a general framework for probabilistic regression-type models with an intractable normalizing constant in this section. Suppose we have independent observations $\bm{y}_1,\cdots,\bm{y}_n\in \mathbb{R}^d$ from unknown continuous distributions $q(\bm{y}|\bm{x}_1),$ $\cdots,q(\bm{y}|\bm{x}_n)$, with associated covariates $\bm{x}_1,\cdots, \bm{x}_n$. Further assume that we have parametrized models $p(\cdot|\bm{x}_i,\bm{\theta})$ for $i=1,\cdots,n$, where $\bm{\theta}\in\mathbb{R}^p$ is a vector of unknown parameters. Therefore, $\bm{y}_1,\cdots,\bm{y}_n$ are independent but not necessarily identically distributed under the fixed design setting for these parametrized models. Throughout the entire paper all random variables are defined on an underlying probability space $\{\Omega,\mathcal{F},\mathbb{P}\}$. Thus, for example, the expectation of $\bm{y}$ may be expressed as $\mathbb{E}(\bm{y})=\int_\Omega \bm{y} d\mathbb{P}$.

Our unified framework for regression-type models with an intractable normalizing constant is given by
\begin{align}\label{model for INNID}
p(\bm{y}|\bm{x}_i,\bm{\theta})=\frac{1}{Z_i(\bm{\theta})}p_0(\bm{y})\widetilde{p}_1(\bm{y}|\bm{x}_i,\bm{\theta}),
\end{align}
where $p_0(\bm{y})$ is a known function of $\bm{y}$ and can be used to indicate additional information of the response, e.g. the range of the response, $\widetilde{p}_1(\bm{y}|\bm{x}_i,\bm{\theta})$ describes the relationship between the response and the covariates, and
\[
Z_i(\bm{\theta})\coloneqq Z(\bm{\theta},\bm{x}_i)=\int p_0(\bm{y})\widetilde{p}_1(\bm{y}|\bm{x}_i,\bm{\theta})d\bm{y},
\]
for $i=1,\cdots,n$ when $\bm{y}$ is continuous. When $\bm{y}$ is discrete, the normalizing constant, $Z_i(\bm{\theta})$, is given by
\[
Z_i(\bm{\theta})=\sum_{\bm{y}\in\mathcal{D}} p_0(\bm{y})\widetilde{p}_1(\bm{y}|\bm{x}_i,\bm{\theta}),
\]
where $\mathcal{D}$ is the domain of the observed data. For example, $\mathcal{D}=\{0,1\}^d$ for binary data and $\mathcal{D}=\{0,1,2,\cdots\}^d$ for count data. Suppose that the model (\ref{model for INNID}) is correctly specified, then the true distribution model of $\bm{y}$ satisfies $q(\bm{y}|\bm{x}_i)\propto p_0(\bm{y})\widetilde{p}_1(\bm{y}|\bm{x}_i,\bm{\theta}_0)$ for $i=1,\cdots,n$, where $\bm{\theta}_0$ denotes the true parameters.

To provide some insight into the general framework of the models defined above, we consider the following three examples.

\begin{exmp}[Truncated Gaussian regression]
If we set $p_0(\bm{y})=\mathbb{1}(\bm{y}>\bm{0})$ with $\mathbb{1}(\cdot)$ being an indicator function and $\bm{y}>\bm{0}$ indicates that each component of $\bm{y}$ is positive, and $\widetilde{p}_1(\bm{y}|\bm{x}_i,\bm{\theta})$ denotes an unnormalized Gaussian distribution with conditional mean $\bm{B}\bm{x}_i$, unknown precision matrix $\bm{\Lambda}$ and $\bm{\theta}=\left({\rm vec}^\top (\bm{B}),{\rm vech}^\top (\bm{\Lambda})\right)^\top$ where ${\rm vec}(\cdot)$ denotes the vectorization and ${\rm vech}(\cdot)$ denotes the half vectorization (\citealt{Magnus2019}). Then the probabilistic model (\ref{model for INNID}) gives a truncated Gaussian regression.
\end{exmp}

\begin{exmp}[Size-biased sampling]
If we treat $p_0(\bm{y})$ as a known weight function, then we get size-biased sampling. In other words, suppose that the original observation $\bm{y}_i$ has $\widetilde{p}_1(\bm{y}|\bm{x}_i,\bm{\theta})$ as the unnormalized probability mass or density function, and the probability of recording the observation $\bm{y}_i$ is $0\leq p_0(\bm{y}_i)\leq 1$, then the distribution function of the recorded observation, $\bm{y}_i^w$, follows the probabilistic model (\ref{model for INNID}) (\citealt{patil1978weighted}).
\end{exmp}

\begin{exmp}[Conway-Maxwell-Poisson regression]\label{exmp::CMP}
If we let $y\in\mathbb{N}_0$ where $\mathbb{N}_0$ is the set of non-negative integers, set $p_0(y)=1$ and 
\[
\widetilde{p}_1(y|\bm{x}_i,\bm{\theta})=\frac{\lambda_i^{y}}{(y!)^\nu},
\]
where $\bm{\theta}=(\bm{\beta}^\top,\nu)^\top$, $\nu\geq 0$ denotes the dispersion parameter and $\lambda_i=\exp (\bm{x}_i^\top\bm{\beta})$ is a generalization of the Poisson mean parameter, then the model (\ref{model for INNID}) gives a CMP regression model with $Z_i(\bm{\theta})=\sum_{s=0}^\infty \frac{\lambda_i^s}{(s!)^\nu}$. The CMP links together three common distributions as special cases: the geometric ($\nu=0$, and $\lambda_i<1$), Poisson ($\nu=1$), and Bernoulli ($\nu\to\infty$). Additionally, when $\nu\in[0,1)$, the CMP describes over-dispersed data relative to a Poisson distribution with the same mean, while when $\nu>1$, the CMP model is appropriate for under-dispersed data (\citealt{shmueli2005useful}).
\end{exmp}

\section{Score Matching and Its Generalization}\label{sect::sm}
In this section, we first present the original score matching for continuous data under the general framework given in Section \ref{sect::model setting} and then propose our generalized score matching approach for discrete data. Finally, we will derive our theoretical results for these estimations. Throughout the paper, for observed $\bm{y}_i=(y_{i1},\cdots,y_{id})^\top$ following a distribution $q(\bm{y}|\bm{x}_i)$ and for any function $l(\bm{y})$, we define the gradient of $l(\bm{y})$ as $\nabla l(\bm{y})=\frac{\partial l(\bm{y})}{\partial \bm{y}}=\left(\frac{\partial l(\bm{y})}{\partial y_1},\cdots,\frac{\partial l(\bm{y})}{\partial y_d}\right)^\top$, $\nabla l(\bm{y}_i)=\frac{\partial l(\bm{y})}{\partial \bm{y}}\big|_{\bm{y}=\bm{y}_i}$ and $\frac{\partial l(\bm{y}_i)}{\partial y_j}=\frac{\partial l(\bm{y})}{\partial y_j}\big|_{\bm{y}=\bm{y}_i}$ for $i=1,\cdots,n$ and $j=1,\cdots,d$.

\subsection{Score matching for continuous INID data}

Based on the Fisher divergence for two probability distributions $q(\bm{y}|\bm{x}_i)$ and $p(\bm{y}|\bm{x}_i,\bm{\theta})$ on $\mathbb{R}^d$, the basic score matching objective function $D_{\rm SM}(q_*,p_*)$ is given by
\begin{align}\label{sm population obj func}
D_{\rm SM}(q_*,p_*)=&\frac{1}{n}\sum_{i=1}^n\mathbb{E} \left( \|\nabla\log q(\bm{y}_i|\bm{x}_i)-\nabla\log p(\bm{y}_i|\bm{x}_i,\bm{\theta})\|^2 \right).
\end{align}
Note that $D_{\rm SM}(q_*,p_*)$ can be decomposed as
\[
D_{\rm SM}(q_*,p_*)=g(q_*)+d_{\rm SM}(q_*,p_*),
\]
where 
\begin{align*}
d_{\rm SM}(q_*,p_*)=\frac{1}{n}\sum_{i=1}^n\mathbb{E}  \left[2\sum_{j=1}^d\frac{\partial^2}{\partial^2 y_{j}}\log p(\bm{y}_i|\bm{x}_i,\bm{\theta})+\sum_{j=1}^d\left\{\frac{\partial}{\partial y_{j}}\log p(\bm{y}_i|\bm{x}_i,\bm{\theta})\right\}^2\right],
\end{align*}
and $g(q_*)$ is a constant depending on $q_*$ but not on $\bm{\theta}$. An empirical estimator of the population function, $d_{\rm SM}(q_*,p_*)$, is given by
\begin{align}\label{emprical sm cts}
\hat{d}_{\rm SM}(\bm{\theta})=\frac{1}{n}\sum_{i=1}^n \rho_{\rm SM}(\bm{y}_i|\bm{x}_i,\bm{\theta}),
\end{align}
where
\begin{align}\label{rho sm}
\rho_{\rm SM}(\bm{y}_i|\bm{x}_i,\bm{\theta})=&2\sum_{j=1}^d\frac{\partial^2}{\partial^2 y_{j}}\log p_0(\bm{y}_i)+2\sum_{j=1}^d\frac{\partial^2}{\partial^2 y_{j}}\log \widetilde{p}_1(\bm{y}_i|\bm{x}_i,\bm{\theta})\nonumber\\
&+\sum_{j=1}^d\left\{\frac{\partial}{\partial y_{j}}\log p_0(\bm{y}_i)+\frac{\partial}{\partial y_{j}}\log \widetilde{p}_1(\bm{y}_i|\bm{x}_i,\bm{\theta})\right\}^2.
\end{align}
The score matching estimator for $\bm{\theta}$ is then defined as
\begin{align}\label{sm est def}
\hat{\bm{\theta}}=\argmin_{\bm{\theta}}\hat{d}_{\rm SM}(\bm{\theta}).
\end{align}

\subsection{Generalized score matching for discrete INID data}
\cite{lyu2012interpretation} noticed that $\nabla \log q(\bm{y}|\bm{x}_i)=\frac{\nabla q(\bm{y}|\bm{x}_i)}{q(\bm{y}|\bm{x}_i)}$ where the gradient is a linear operator. \cite{lyu2012interpretation} then proposed a generalization of score matching in which a general linear operator $\mathcal{L}$ replaces the gradient operator. Thus, the generalized score matching objective function $D_{\mathcal{L}}(q_*,p_*)$ for $q_*=\prod_{i=1}^n q(\bm{y}|\bm{x}_i)$ and $p_*=\prod_{i=1}^n p(\bm{y}|\bm{x}_i,\bm{\theta})$ is given by
\begin{align}\label{original gsm loss}
D_{\mathcal{L}}(q_*,p_*)=&\frac{1}{n}\sum_{i=1}^n\mathbb{E} \left(\left\|\frac{\mathcal{L} q(\bm{y}|\bm{x}_i)}{q(\bm{y}|\bm{x}_i)}-\frac{\mathcal{L}  p(\bm{y}|\bm{x}_i,\bm{\theta})}{p(\bm{y}|\bm{x}_i,\bm{\theta})}\right\|^2\right),
\end{align}
where $\|\cdot\|$ denotes the Euclidean norm. 

To deal with the discrete data, \cite{lyu2012interpretation} studied a special case of the linear operator $\mathcal{L}$ in (\ref{original gsm loss}) which is called the marginalization operator $\mathcal{M}$. This special operator is defined to be $\mathcal{M}q(\bm{y}|\bm{x}_i)\coloneqq(\mathcal{M}_1q(\bm{y}|\bm{x}_i),\cdots,\mathcal{M}_dq(\bm{y}|\bm{x}_i))^\top$ with $\mathcal{M}_jq(\bm{y}|\bm{x}_i)\coloneqq\sum_{y_{j}}q(\bm{y}|\bm{x}_i)$ for $i=1,\cdots,n$ and $j=1,\cdots,d$. It is worth noting that when the dimension $d$ equals one, the proposed generalized score matching would still depend on the explicit computation of the intractable normalization constant. A similar problem occurs in ratio matching and this approach also fails to deal with univariate data. Therefore, we propose a novel linear operator which is called forward difference in this section.

For our proposed generalized score matching for regression-type models for discrete data, we first introduce the estimation method considering univariate responses in Section \ref{subsection::sm discrete}. Then the estimation approach is extended to the  multivariate response case in Section \ref{subsection::sm discrete multi} and discussed in more detail in Section \ref{sup sec::gsm multi discrete} of the supplementary material.

\subsubsection{Generalized score matching for discrete INID univariate response}\label{subsection::sm discrete}
Score matching for continuous data is designed to compare the slopes of the logarithms of the densities. However, the proposed variations of score matching for discrete data fail to explore such relationships between the true density and the parametrized density. We consider a novel linear operator $\mathcal{L}$ which is defined as $\mathcal{L}p(y|\bm{x}_i,\bm{\theta})\coloneqq p(y^+|\bm{x}_i,\bm{\theta})-p(y|\bm{x}_i,\bm{\theta})$ for $i=1,\cdots,n$, where $y^+=y+1$. If the range of $y$ is bounded, let $p(y^+|\bm{x}_i,\bm{\theta})=0$ when $y$ is located at the boundary. Similarly, let $y^-=y-1$ and $p(y^-|\bm{x}_i,\bm{\theta})=0$ at the boundary of the response domain. Note that this linear operator gives a discrete analogue of the slope of $p(y|\bm{x}_i,\bm{\theta})$ at the point $y$. After omitting the constant in $\mathcal{L}p(y|\bm{x}_i,\bm{\theta})/p(y|\bm{x}_i,\bm{\theta})=p(y+1|\bm{x}_i,\bm{\theta})/p(y|\bm{x}_i,\bm{\theta})-1$, we can see that the basic principle in our method is to force the ratio $p(y^+|\bm{x}_i,\bm{\theta})/p(y|\bm{x}_i,\bm{\theta})$ to be as close as possible to the corresponding ratio given by the data, i.e., $q(y^+|\bm{x}_i)/q(y|\bm{x}_i)$.

To avoid the zero denominator in the slopes, we consider the following transformation (\citealt{hyvarinen2007some}) of the slopes:
\[
t(u)=\frac{1}{1+u}.
\]
Now, any probability that is zero and leads to a ratio that is infinite will give a value of $t(\infty)=0$ for this transformation.

Therefore, we propose that the model is estimated by minimizing the following objective function for $q_*$ and $p_*$:
\begin{align}\label{discrete overall population function}
D_{\rm GSM}(q_*,p_*)=\frac{1}{n}\sum_{i=1}^n D_{\rm GSM}(q_i,p_i),
\end{align}
where
\begin{align}\label{discrete population function}
D_{\rm GSM}(q_i,p_i)=&\mathbb{E}  \Bigg\{\left[ t\left( \frac{p(y_i^+|\bm{x}_i,\bm{\theta})}{p(y_i|\bm{x}_i,\bm{\theta})} \right) - t\left( \frac{q(y_i^+|\bm{x}_i)}{q(y|\bm{x}_i)} \right) \right]^2 \nonumber\\
&+ \left[ t\left( \frac{p(y_i|\bm{x}_i,\bm{\theta})}{p(y_i^-|\bm{x}_i,\bm{\theta})}\right) - t\left( \frac{q(y_i|\bm{x}_i)}{q(y_i^-|\bm{x}_i)} \right) \right]^2 \Bigg\},
\end{align}
and $q_i$ and $p_i$ denote $q(y|\bm{x}_i)$ and $p(y|\bm{x}_i,\bm{\theta})$, respectively. The following theorem will show that (\ref{discrete population function}) is tractable. Its proof, as with all proofs in this article, is given in the Appendix.
\begin{theorem}\label{thm:: tractable obj funct for GSM}
The $i$-th component of the overall population objective function (\ref{discrete overall population function}), (\ref{discrete population function}), can be decomposed as
\begin{align*}
D_{\rm GSM}(q_i,p_i)=&\mathbb{E} \Bigg\{ t\left( \frac{p(y_i^+|\bm{x}_i,\bm{\theta})}{p(y_i|\bm{x}_i,\bm{\theta})} \right)^2+ t\left( \frac{p(y_i|\bm{x}_i,\bm{\theta})}{p(y_i^-|\bm{x}_i,\bm{\theta})}\right)^2-2t\left( \frac{p(y_i^+|\bm{x}_i,\bm{\theta})}{p(y_i|\bm{x}_i,\bm{\theta})} \right) \Bigg\}+C,
\end{align*}
where $C$ is a constant not depending on $\bm{\theta}$.
\end{theorem}

It is worth noting that Theorem \ref{thm:: tractable obj funct for GSM} indicates that $D_{\rm GSM}(q_*,p_*)$ can be decomposed as
\[
D_{\rm GSM}(q_*,p_*)=g(q_*)+d_{\rm GSM}(q_*,p_*),
\]
where
\[
d_{\rm GSM}(q_*,p_*)=\frac{1}{n}\sum_{i=1}^n \mathbb{E} \Bigg\{ t\left( \frac{p(y_i^+|\bm{x}_i,\bm{\theta})}{p(y_i|\bm{x}_i,\bm{\theta})} \right)^2 + t\left( \frac{p(y_i|\bm{x}_i,\bm{\theta})}{p(y_i^-|\bm{x}_i,\bm{\theta})}\right)^2-2t\left( \frac{p(y_i^+|\bm{x}_i,\bm{\theta})}{p(y_i|\bm{x}_i,\bm{\theta})} \right) \Bigg\},
\]
and $g(q_*)$ is a constant depending on $q_*$ but not on $\bm{\theta}$. An empirical estimator of the population function, $d_{\rm GSM}(q_*,p_*)$, is given by
\begin{align}\label{discrete empirical function}
\hat{d}_{\rm GSM}(\bm{\theta})=\frac{1}{n}\sum_{i=1}^n \rho_{\rm GSM}(y_i|\bm{x}_i,\bm{\theta}),
\end{align}
where
\begin{align}\label{rho gsm}
\rho_{\rm GSM}(y_i|\bm{x}_i,\bm{\theta})= t\left( \frac{p(y_i^+|\bm{x}_i,\bm{\theta})}{p(y_i|\bm{x}_i,\bm{\theta})} \right)^2 + t\left( \frac{p(y_i|\bm{x}_i,\bm{\theta})}{p(y_i^-|\bm{x}_i,\bm{\theta})}\right)^2-2t\left( \frac{p(y_i^+|\bm{x}_i,\bm{\theta})}{p(y_i|\bm{x}_i,\bm{\theta})} \right).
\end{align}
The generalized score matching estimator for $\bm{\theta}$ is then defined as
\begin{align}\label{gsm est def}
\hat{\bm{\theta}}=\argmin_{\bm{\theta}}\hat{d}_{\rm GSM}(\bm{\theta}).
\end{align}

As for the local consistency, we have a result which is analogous to the consistency theorem in \cite{hyvarinen2007some}. This is given by the following theorem.
\begin{theorem}\label{thm:: GSM unique solution}
Assume that the model $p(y|\bm{x}_i,\bm{\theta})$ is correct, that is, there exits a $\bm{\theta}_0$ such that $p(y|\bm{x}_i,\bm{\theta}_0)=q(y|\bm{x}_i)$ for $i=1,\cdots,n$ and further suppose that the model is identifiable, i.e. for each $\bm{\theta}\neq\bm{\theta}_0$, there exists a set of $\bm{y}$ of positive probability under $p(\bm{y}|\bm{x}_i,\bm{\theta}_0)$ such that $p(\bm{y}|\bm{x}_i,\bm{\theta})\neq p(\bm{y}|\bm{x}_i,\bm{\theta}_0)$. Then, $D_{\rm GSM}(q_*,p_*)=0$ if and only if $\bm{\theta}=\bm{\theta}_0$, where $D_{\rm GSM}$ is defined in (\ref{discrete overall population function}).
\end{theorem}

\subsubsection{Generalized score matching for discrete INID multivariate response}\label{subsection::sm discrete multi}
To extend our proposed generalized score matching to multivariate cases, we consider a linear operator $\mathcal{L}$ which is defined as
\[
\mathcal{L}p(\bm{y}|\bm{x}_i,\bm{\theta})=\begin{pmatrix}
\vdots\\
\mathcal{L}_j p(\bm{y}|\bm{x}_i,\bm{\theta})\\
\vdots
\end{pmatrix}=\begin{pmatrix}
\vdots\\
p(\bm{y}^{(j_{+})}|\bm{x}_i,\bm{\theta})-p(\bm{y}|\bm{x}_i,\bm{\theta})\\
\vdots
\end{pmatrix},
\]
where $\bm{y}^{(j_{+})}=(y_{1},\cdots,y_{j}+1,\cdots,y_{d})^\top$ for $j=1,\cdots,d$. If the range of $\bm{y}$ is bounded, let $p(\bm{y}^{(j_{+})}|\bm{x}_i,\bm{\theta})=0$ when $\bm{y}$ is located at the boundary. Similarly, let $\bm{y}^{(j_{-})}=(y_{1},\cdots,y_{j}-1,\cdots,y_{d})^\top$ and $p(\bm{y}^{(j_{-})}|\bm{x}_i,\bm{\theta})=0$ at the boundary of the response domain. After using the same transformation in the univariate case, the $i$-th component of the overall population objective function is given by
\begin{align*}
D_{\rm GSM}(q_i,p_i)=&\mathbb{E} \Bigg(\sum_{j=1}^d \Bigg\{\left[ t\left( \frac{p(\bm{y}_i^{(j_{+})}|\bm{x}_i,\bm{\theta})}{p(\bm{y}_i|\bm{x}_i,\bm{\theta})} \right) - t\left( \frac{q(\bm{y}_i^{(j_{+})}|\bm{x}_i)}{q(\bm{y}_i|\bm{x}_i)} \right) \right]^2 \nonumber\\
&+ \left[ t\left( \frac{p(\bm{y}_i|\bm{x}_i,\bm{\theta})}{p(\bm{y}_i^{(j_{-})}|\bm{x}_i,\bm{\theta})}\right) - t\left( \frac{q(\bm{y}_i|\bm{x}_i)}{q(\bm{y}_i^{(j_{-})}|\bm{x}_i)} \right) \right]^2 \Bigg\}\Bigg).
\end{align*}
In the setting of discrete data with multivariate response, the proposed generalized score matching estimation has analogous theoretical properties which is studied in the univariate case. Detailed discussion of these properties can be found in Section \ref{sup sec::gsm multi discrete} of the supplementary material.

\subsection{Theoretical properties of generalized score matching}\label{sect:: thm property}
Note that our proposed generalized score matching estimators given by (\ref{sm est def}) and (\ref{gsm est def}) are M-estimators. Thus we can apply similar techniques to those used in deriving theoretical results of M-estimators (\citealt{van2000asymptotic}). 

Before discussing the limiting behaviors of our proposed score matching estimators for regression-type models, we state the following conditions for deriving asymptotics of our score matching estimator. For the sake of simplification, we present conditions for generalized score matching for discrete data with multivariate response. For univariate response, we can change the bold notation $\bm{y}_i$ to $y_i$. In the following, denote the parameter space for $\bm{\theta}$ by $\Theta$ and let $\mathcal{I}_n(\bm{\theta})=-\mathbb{E}\left( \frac{\partial^2 \hat{d}(\bm{\theta})}{\partial \bm{\theta} \partial \bm{\theta}^\top}\right)$ and $\mathcal{J}_n(\bm{\theta})=n{\rm Cov}\left(\frac{\partial \hat{d}(\bm{\theta})}{\partial \bm{\theta}}\right)-\mathcal{I}_n(\bm{\theta})$ in each case $\hat{d}=\hat{d}_{\rm SM}$ and $\hat{d}=\hat{d}_{\rm GSM}$ where $\hat{d}_{\rm SM}$ and $\hat{d}_{GSM}$ are defined in (\ref{emprical sm cts}) and (\ref{discrete empirical function}), respectively. Additionally, we consider the corresponding cases $\rho=\rho_{\rm SM}$ and $\rho=\rho_{\rm GSM}$ where $\rho_{\rm SM}$ and $\rho_{\rm GSM}$ are given in (\ref{rho sm}) and (\ref{rho gsm}), respectively.

\begin{description}
\item[(C1)] There exists an open subset $\mathcal{B}$ of $\Theta$ that contains the true parameter point $\bm{\theta}_0$ such that for almost all $(\bm{y}_i,\bm{x}_i)$, $\rho(\bm{y}_i|\bm{x}_i,\bm{\theta})$ admits the first derivatives $\frac{\partial \rho(\bm{y}_i|\bm{x}_i,\bm{\theta})}{\partial \bm{\theta}}$ and $\mathbb{E}( \rho(\bm{y}_i|\bm{x}_i,\bm{\theta}))<\infty$ for all $\bm{\theta}\in\mathcal{B}$. Furthermore, for $i=1,\cdots,n$,
\begin{align*}
\sup_{\bm{\theta}\in\mathcal{B}}\max_{k}\left| \frac{\partial \rho(\bm{y}_i|\bm{x}_i,\bm{\theta})}{\partial \theta_{k}} \right| &\leq M_1(\bm{y}_i),
\end{align*}
where $m_1=\mathbb{E}[|M_1(\bm{y}_i)|]<\infty$;

\item[(C2)] $\mathcal{I}_n(\bm{\theta}_0)\to\mathcal{I}(\bm{\theta}_0)$ and $\mathcal{J}_n(\bm{\theta}_0)\to\mathcal{J}(\bm{\theta}_0)$ as $n\to\infty$. We assume that $\mathcal{I}(\bm{\theta}_0)$ and $\mathcal{I}(\bm{\theta}_0)+\mathcal{J}(\bm{\theta}_0)$ are positive definite;

\item[(C3)] There exists a $\delta>0$ such that $\widetilde{M}(\bm{x}_i,\bm{\theta}_0)\coloneqq \mathbb{E}\left[\left\|\frac{\partial \rho(\bm{y}_i|\bm{x}_i,\bm{\theta}_0)}{\partial \bm{\theta}}\right\|^{2+\delta}\right]$ satisfies
\[
\lim_{n\to \infty}n^{-1}\sum_{i=1}^n \widetilde{M}(\bm{x}_i,\bm{\theta}_0)=0;
\]

\item[(C4)] There exists an open subset $\mathcal{B}$ of $\Theta$ that contains the true parameter point $\bm{\theta}_0$ such that for almost all $(\bm{y}_i,\bm{x}_i)$, $\rho(\bm{y}_i|\bm{x}_i,\bm{\theta})$ admits all third derivatives $\frac{\partial^3 \rho(\bm{y}_i|\bm{x}_i,\bm{\theta})}{\partial \theta_{j_1}\partial \theta_{j_2}\partial \theta_{j_3}}$ for all $\bm{\theta}\in\mathcal{B}$. Furthermore, for $i=1,\cdots,n$, 
\begin{align*}
\sup_{\bm{\theta}\in\mathcal{B}} \max_{j_1,j_2,j_3}\left| \frac{\partial^3 \rho(\bm{y}_i|\bm{x}_i,\bm{\theta})}{\partial \theta_{j_1}\partial \theta_{j_2}\partial \theta_{j_3}} \right| &\leq M_2(\bm{y}_i),
\end{align*}
where $m_2=\mathbb{E}[|M_2(\bm{y}_i)|^{2+\delta}]<\infty$ for some $\delta>0$.

\end{description}

Condition (C1) comes from the differentiation lemma in \cite{klenke2013probability} which ensures the interchange of integration and differentiation in an open neighborhood around $\bm{\theta}_0$. Condition (C2) is a standard condition for establishing the convergence of the Fisher information matrix and the covariance of the score functions. Condition (C3) is a Lyapounov condition which, in conjunction with Condition (C4), is commonly used in asymptotics for MLEs under the INID setting (\citealt{lee1998analysis,van2000asymptotic}).

Our first result in this section shows that the expectation of the score vector is equal to $\bm{0}$ when $\bm{\theta}=\bm{\theta}_0$. Note that $\bm{\theta}_0$ is the population value of $\bm{\theta}$ assuming that the parametric model is correct. 

\begin{proposition}\label{thm:: sm score function expectation}
Suppose that Condition (C1) in Section \ref{sect:: thm property} holds. Then
\[
\mathbb{E}\left(\frac{\partial \hat{d}(\bm{\theta}_0)}{\partial \bm{\theta}}\Bigg|_{\bm{\theta}=\bm{\theta}_0}\right)=\bm{0},
\]
in each case $\hat{d}=\hat{d}_{\rm SM}$ and $\hat{d}=\hat{d}_{\rm GSM}$ where $\hat{d}_{\rm SM}$ and $\hat{d}_{GSM}$ are defined in (\ref{emprical sm cts}) and (\ref{discrete empirical function}), respectively, $\bm{\theta}_0$ denotes the population parameter vector and the expectation is taken with respect to the population distribution.
\end{proposition}

The above result gives 
\begin{align*}
{\rm Cov}\left(\frac{\partial \hat{d}(\bm{\theta}_0)}{\partial \bm{\theta}}\right)=\frac{1}{n^2}\sum_{i=1}^n \mathbb{E}\left\{\frac{\partial \rho(\bm{y}_i|\bm{x}_i,\bm{\theta}_0)}{\partial \bm{\theta}}\frac{\partial \rho(\bm{y}_i|\bm{x}_i,\bm{\theta}_0)}{\partial \bm{\theta}^\top}\right\},
\end{align*}
and $\mathcal{J}_n(\bm{\theta})$ can be represented as
\begin{align*}
\mathcal{J}_n(\bm{\theta})=&\frac{1}{n}\sum_{i=1}^n\mathbb{E}\left\{\frac{\partial \rho(\bm{y}_i|\bm{x}_i,\bm{\theta})}{\partial \bm{\theta}}\frac{\partial \rho(\bm{y}_i|\bm{x}_i,\bm{\theta})}{\partial \bm{\theta}^\top}+\frac{\partial^2 \rho(\bm{y}_i|\bm{x}_i,\bm{\theta})}{\partial \bm{\theta} \partial \bm{\theta}^\top}\right\}.
\end{align*}

\begin{proposition}\label{prop::clt score function}
Suppose that Conditions (C1)-(C4) in Section \ref{sect:: thm property} hold. Then
\[
\sqrt{n}\frac{\partial \hat{d}(\bm{\theta}_0)}{\partial \bm{\theta}}\stackrel{d}{\longrightarrow}N\left(\bm{0},\mathcal{I}(\bm{\theta}_0)+\mathcal{J}(\bm{\theta}_0)\right),
\]
and
\[
-\frac{\partial^2 \hat{d}(\bm{\theta}_0)}{\partial \bm{\theta} \partial \bm{\theta}^\top}\stackrel{p}{\longrightarrow}\mathcal{I}(\bm{\theta}_0),
\]
in each case $\hat{d}=\hat{d}_{\rm SM}$ and $\hat{d}=\hat{d}_{\rm GSM}$ where $\hat{d}_{\rm SM}$ and $\hat{d}_{GSM}$ are defined in (\ref{emprical sm cts}) and (\ref{discrete empirical function}), respectively, $\mathcal{I}(\bm{\theta}_0)$ and $\mathcal{J}(\bm{\theta}_0)$ which are defined in Condition (C2) in Section \ref{sect:: thm property} are the limiting of matrices $\mathcal{I}_n(\bm{\theta}_0)$ and $\mathcal{J}_n(\bm{\theta}_0)$, respectively.
\end{proposition}

Based on Proposition \ref{prop::clt score function}, which demonstrates the central limit theorem of the score function and the weak law of large numbers of the Hessian matrix, we are now able to build a central limit theorem for the score matching estimator $\hat{\bm{\theta}}$. The proof is very similar to that of Theorem 1 in \cite{fan2001variable} and is omitted.

\begin{theorem}\label{thm::clt cts sm}
Suppose that Conditions (C1)-(C4) in Section \ref{sect:: thm property} hold. Then
\[
\sqrt{n}(\hat{\bm{\theta}}-\bm{\theta}_0)\stackrel{d}{\longrightarrow} N\left(\bm{0},\mathcal{I}^{-1}(\bm{\theta}_0)+\mathcal{I}^{-1}(\bm{\theta}_0)\mathcal{J}(\bm{\theta}_0)\mathcal{I}^{-1}(\bm{\theta}_0) \right),
\]
where $\hat{\bm{\theta}}$ is given by (\ref{sm est def}) or (\ref{gsm est def}).
\end{theorem}

In practice, both $\mathcal{I}(\bm{\theta}_0)$ and $\mathcal{J}(\bm{\theta}_0)$ are unknown. To make the above theorem useful, one needs to find consistent estimators of these matrices. Using the fact that $\mathcal{I}_n(\bm{\theta}_0)\to \mathcal{I}(\bm{\theta}_0)$ and $\mathcal{J}_n(\bm{\theta}_0)\to \mathcal{J}(\bm{\theta}_0)$, we can show that the asymptotic covariance matrix $\mathcal{I}^{-1}(\bm{\theta}_0)+\mathcal{I}^{-1}(\bm{\theta}_0)\mathcal{J}(\bm{\theta}_0)\mathcal{I}^{-1}(\bm{\theta}_0)$ can be consistently estimated by $\mathcal{I}^{-1}_n(\hat{\bm{\theta}})+\mathcal{I}_n^{-1}(\hat{\bm{\theta}})\mathcal{J}_n(\hat{\bm{\theta}})\mathcal{I}_n^{-1}(\hat{\bm{\theta}})$. Additionally, it would be more convenient in practice to use $\hat{\mathcal{I}}_n(\hat{\bm{\theta}})$ and $\hat{\mathcal{I}}_n(\hat{\bm{\theta}})$ which are the sample analogues of $\mathcal{I}_n(\hat{\bm{\theta}})$ and $\mathcal{J}_n(\hat{\bm{\theta}})$, respectively.

\section{Case Study: Doctoral Publication Data Analysis}\label{sect::case study}

We analyzed data from \cite{long1990origins} on the number of publications produced by Ph.D. biochemists. This dataset aims to examine the relationship between the number of Ph.D. scientist's publications and the gender (coded one for female), the marriage status (coded one if married), the number of children under age six (\emph{kid5}), the prestige of Ph.D. program (\emph{phd}) and the number of articles by mentor in last three years (\emph{mentor}). We focused only on the students with at least one publication and the data consists of 640 Ph.D. candidates and the average number of publications is 2.42. The number of publications exhibits strong over-dispersion (\citealt{long2006regression}). We compared fits to the original data using the CMP regression model with approximate MLE and generalized score matching estimation. Both models were fit using a Windows desktop with an AMD Ryzen 9 5900X CPU running at 3.7 GHz and 32 GB RAM. The approximate MLE was obtained by using the function \code{glm.cmp} from the R package \pkg{COMPoissonReg} accompanying the paper by \cite{sellers2010flexible} and the generalized score matching estimator was performed using the function \code{optim} in the R package \pkg{stats} (\citealt{R-Core-Team:2013wf}) with the Nelder-Mead algorithm (\citealt{nelder1965simplex}).

As seen in Example \ref{exmp::CMP}, the normalizing constant is an infinite sum which is intractable. Therefore, within the maximum likelihood framework, several approximation approaches have been proposed, i.e., use of a truncation  (\citealt{minka2003computing}) and an asymptotic approximation (\citealt{shmueli2005useful}) of the normalization constant series. We refer to such estimation approach as approximate MLE. However, these approximations will become inaccurate under some situations. For example, the asymptotic approximation is accurate only if $\lambda_i>10^\nu$. It should be emphasized that the function \code{glm.cmp} uses a hybrid method that includes the truncation and asymptotic approximations of the normalizing constant.

The estimated coefficients, standard errors (SEs), absolute t-statistics, dispersion and computer run times from the CMP regression model with approximate MLE and score matching estimation are given in Table \ref{cmp::case study2}. It is worth noting that the SEs of generalized score matching estimates are obtained based on the consistent estimator of the asymptotic variance, which is carefully discussed in Section \ref{sect:: thm property}. 
\begin{table}[H]
\caption{Comparison of the approximate MLE and generalized score matching estimations for the CMP regression model with the doctoral publication dataset.\label{cmp::case study2}}
\centering

\begin{tabular}{ll | rrr | rrr}
  \hline
&&\multicolumn{3}{c|}{Approximate MLE}  & \multicolumn{3}{c}{Generalized score matching} \\
& Coefficient & Estimate & SE & $|t|$ & Estimate & SE & $|t|$ \\ 
   \hline
& intercept            &   -0.3345 &0.0712  & 4.6955 &-0.3141 &0.1022& 3.0736\\
& gender(Female).        &0.0097 &0.0594 & 0.1633 &-0.0893 & 0.0749 &  1.1931 \\
& marriage(Married)    & 0.0993& 0.0678 & 1.4654& 0.0445 & 0.0844 & 0.5268 \\
& kid5                  &-0.0726 &0.0339 & 2.1388   &  -0.0705 &0.0421 & 1.6747 \\
& phd                 & -0.0132 &0.0281 & 0.4687 & 0.0693 & 0.0394  & 1.7583  \\
& mentor                 & 0.1553 &0.0204 & 7.6252 & 0.0830 & 0.0347  & 2.3925  \\
\hline
&dispersion &   0.3698 &0.0402 & & 0.2564 &  0.0827 &\\
& run time (seconds) && 1.55   &&&0.70 &
\end{tabular}
\end{table}

For the standard CMP model, the interpretation of coefficients is opaque (\citealt{Huang2017}), but we can ascertain the direction and statistical significance of the effect of each covariate from the fitted model. For example, Ph.D. candidates enrolled in the program with higher prestige tend to have significantly more publications. Even though the standard errors of approximate MLE in Table \ref{cmp::case study2} are smaller than the ones of generalized score matching, the inference produced by approximate MLE is not reliable. To study the standard errors of these two estimation methods, we first constructed confidence intervals based on the SE and the results are plotted in blue in Figure \ref{fig::case study ci}. The confidence interval of an estimated coefficient $\hat{\theta}_j$ based on its SE $s_j$ is given by $(\hat{\theta}_j-1.96s_j,\hat{\theta}_j+1.96s_j)$ for $j=1,\cdots,7$, where the estimated coefficients and SEs are reported in Table \ref{cmp::case study2}. We then generated $1,000$ bootstrap samples and computed the $95\%$ parametric bootstrap confidence intervals based on bootstrap percentile (\citealt{Efron1994}), whose results are given in Figure \ref{fig::case study ci} and plotted in red. It can be readily seen that, in Figure \ref{fig::case study ci}, the bootstrap confidence intervals and the confidence intervals based on SE overlap for generalized score matching. However, the length of the bootstrap confidence intervals is always larger than the one of the confidence intervals based on SE for approximate MLE. Figure \ref{fig::case study ci} indicates that the approximate MLE method underestimates the variance of the parameters and yields biased estimates.

\begin{figure}
    \centering
    \begin{minipage}{0.5\textwidth}
        \centering
        \includegraphics[width=1\textwidth]{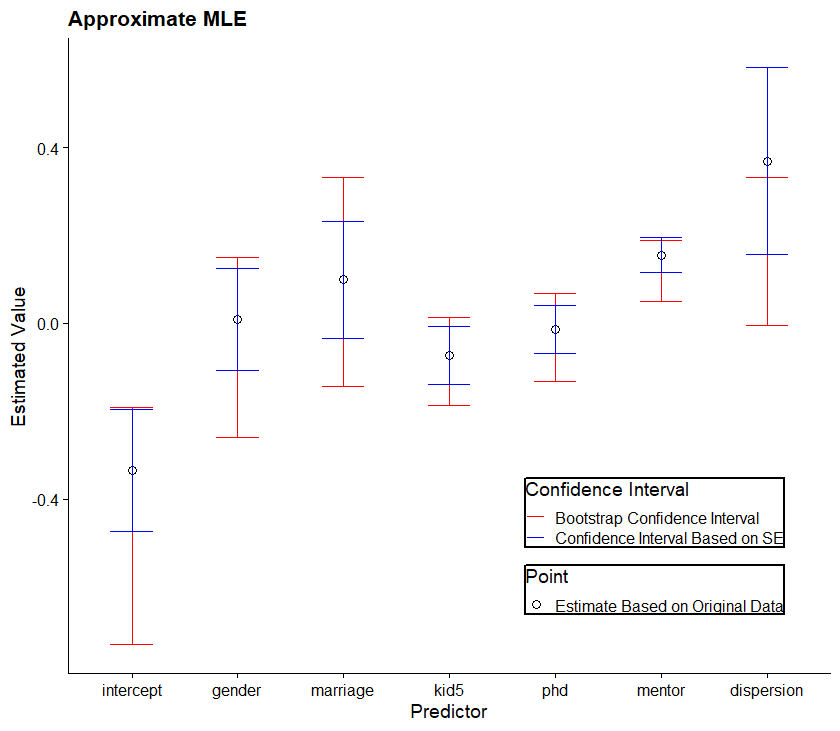} 
    \end{minipage}\hfill
    \begin{minipage}{0.5\textwidth}
        \centering
        \includegraphics[width=1\textwidth]{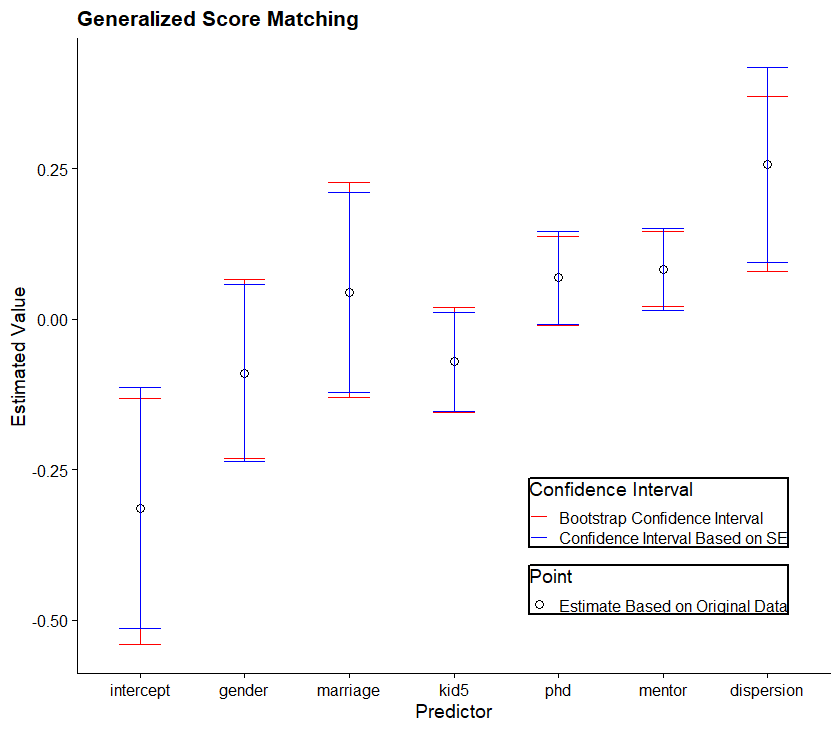} 
    \end{minipage}
    \caption{Plots of $95\%$ parametric bootstrap confidence intervals and confidence intervals based on SE for fitted CMP regression models by (a) approximate MLE and (b) generalized score matching.\label{fig::case study ci}}
\end{figure}

In this data set, the range of the estimated $\lambda_i$ calculated by the \code{glm.cmp} function is from 0.4885 to 1.2223 while the one of the generalized score matching estimated $\lambda_i$ is from 0.5019 to 1.2073. It is worth noting that the relationship between the estimated $\lambda_i$ and the estimated $\nu$ will give inaccurate approximate MLE (\citealt{shmueli2005useful}), which will be indicated by the following simulation results for the CMP regression model. Additionally, the generalized score matching approach took only 0.7 seconds, which is much faster than using approximate MLE to fit the model.

\begin{figure}
    \centering
    \includegraphics[scale=0.3]{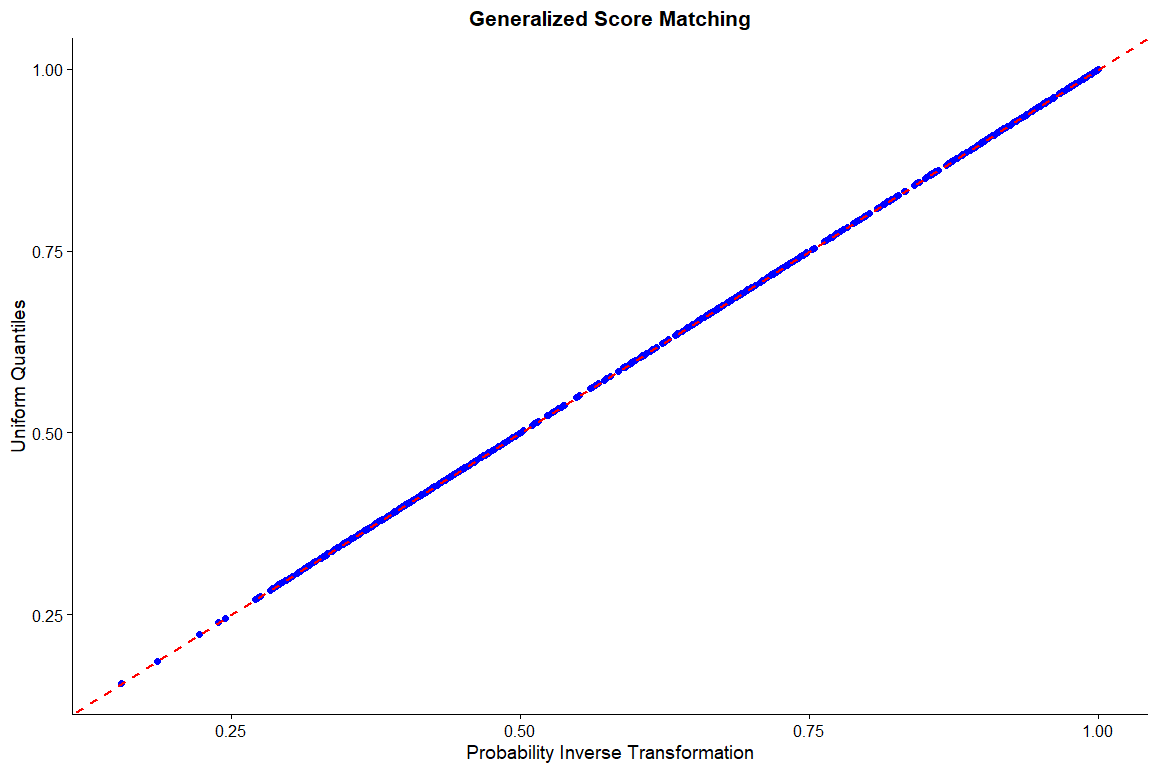} 
    \caption{PIT-uniform quantile plot for the fitted CMP regression model by generalized score matching.\label{fig::case study}}
\end{figure}

For model diagnostics, a PIT-uniform quantile plot is given in Figure \ref{fig::case study}. This plot shows reasonable closeness to uniformity, which indicates that the fitted CMP regression model using generalized score matching is appropriate. To examine the prediction accuracy of the fitted models using our generalized score matching estimators and approximate MLE, we randomly split the whole data set into training and test data sets where the training set contained 448 samples and the test set contained 192 observations. The predicted number of publications was obtained by taking the mean of simulated data from the CMP model with estimated parameters. It is worth noting that simulating data from a CMP by using the approximation of the normalizing constant is not proper in our numerical study as the approximation may be inaccurate. We decided to use the rejection sampling algorithm by \cite{chanialidis2018efficient}, which is an exact method to generate CMP data. The test MSE for the model fitted by generalized score matching is 2.64 and the test MSE for the model using the approximate MLE estimation is 3.90. The comparison of the test MSEs indicate better performance of our proposed generalized score matching method.

\section{Simulation Studies}\label{sect::numerical study}
We conducted a numerical study to evaluate the performance of the score matching estimator for continuous and discrete data. We are particularly interested in examining the bias, standard deviation and the root mean squared error of the score matching estimator. For continuous data, we conducted a simulation study based on truncated Gaussian regression models. For discrete data, we constructed a simulation based on the dataset discussed in Section \ref{sect::case study} to compare our generalized score matching approach with the approximate MLE approach. 

In both settings of the simulations, the sample size $n$ varied in $\{200,500,1000\}$. Additionally, all simulations were conducted via $1,000$ replicates. For the purpose of assessing the performance of parameter estimators, we denote $\hat{\bm{\theta}}^{(k)}$ as the vector estimation of $\bm{\theta}$ in the $k$-th replicate. For each component of $\bm{\theta}$, which is $\theta_j$, the averaged bias of $\hat{\theta}_j^{(k)}$, $k=1,\cdots,1,000$, is $\textrm{BIAS}=\frac{1}{1000}\sum_{k}(\hat{\theta}_j^{(k)}-\theta_j)$, and the standard deviation of $\hat{\theta}_j^{(k)}$ is $\textrm{SD}=\Big\{\frac{1}{1000}\sum_{k_1}(\hat{\theta}_j^{(k_1)}-\frac{1}{1000}\sum_{k_2}\hat{\theta}_j^{(k_2)})^2\Big\}^\frac{1}{2}$. Therefore, the root mean squared error is $\textrm{RMSE}=\sqrt{\textrm{SD}^2+\textrm{BIAS}^2}$. To compare SD with the asymptotic standard deviation of the estimators, we consider a sample version of the asymptotic standard deviation which is a consistent estimation of the intractable asymptotic standard deviation. This sample version is denoted by ASD and the details are carefully discussed in Section \ref{sect:: thm property}. Furthermore, we use BIAS(SM) and BIAS(AMLE) to denote the average bias of the generalized score matching estimator and the approximate MLE, respectively. Similarly, RSME(SM) and RSME(AMLE) are used to denote the root mean squared error of the generalized score matching estimator and the approximate MLE, respectively.

\subsection{Setting 1: simulation for truncated Gaussian regression model}
We simulated data from a truncated Gaussian regression model as follow. For $i=1,\cdots,n$, consider the $2\times 1$ covariate vector $\bm{x}_i$ with $\bm{x}_i=(x_{i1},x_{i2})^\top$, $x_{i1}\equiv 1$, $x_{i2}$ being independent and identically generated from the standard normal distribution $N(0,1)$, and their corresponding regression parameters are ${\rm vec}(\bm{B}_0)=(B_{11},B_{21},B_{12},B_{22})^\top=(1,0.4,-0.5,0.2)^\top$. The true precision matrix $\bm{\Lambda}_0$ is set to be ${\rm vec}(\bm{\Lambda}_0)=(\Lambda_{11},\Lambda_{21},\Lambda_{12},\Lambda_{22})^\top$ $=(20,10,10,30)^\top$. It is worth noting that the covariate matrix is fixed across the replications. The domain of the response $\bm{y}_i=(y_{i1},y_{i2})^\top$ is chosen to be $\mathbb{R}_{>0}^2$ where $\mathbb{R}_{>0}$ denotes the set of positive real numbers. We used the rejection algorithm (\citealt{wilhelm2010tmvtnorm}) to generate the data. That is, we continued generating candidate $\bm{y}_i$ from the bivariate Gaussian distribution $N(\bm{B}_0\bm{x}_i,\bm{\Lambda}_0^{-1})$ until the candidate located inside the support region $\mathbb{R}_{>0}^2$. Figure \ref{fig:first} shows random samples from the bivariate Gaussian distribution $N(\bm{B}_0\bm{x}_i,\bm{\Lambda}_0^{-1})$ and the samples from the truncated Gaussian regression model are located inside the red box.

\begin{figure}
\begin{center}
\includegraphics[scale=0.4]{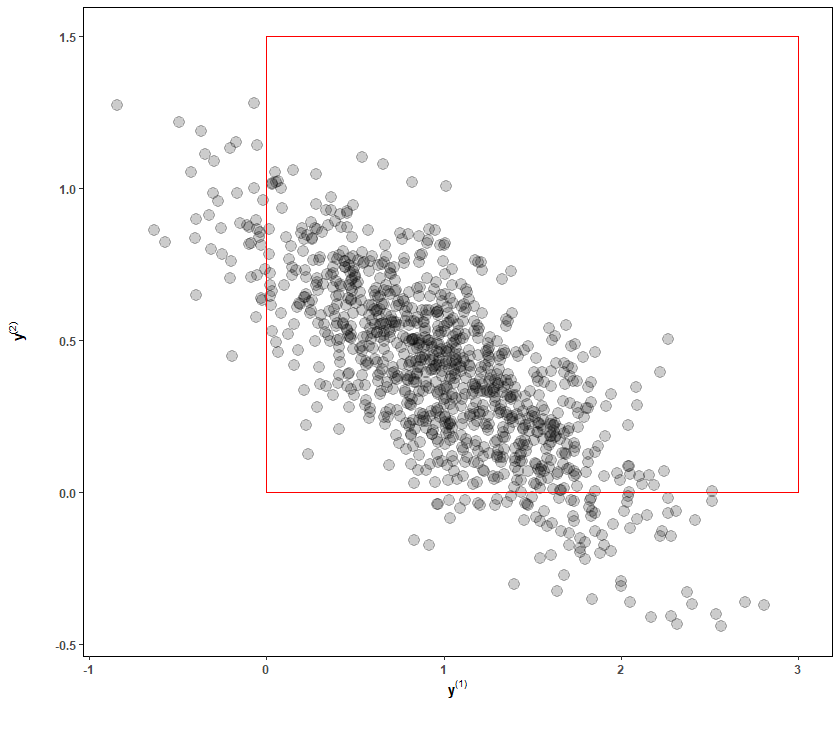}
\end{center}
\caption{Random samples from a truncated Gaussian regression model generated by the rejection algorithm with $\bm{y}^{(1)}=(y_{11},\cdots,y_{n1})^\top$ and $\bm{y}^{(2)}=(y_{12},\cdots,y_{n2})^\top$\label{fig:first}}
\end{figure}

Our parameter of interest is $\bm{\theta}=\left({\rm vec}^\top(\bm{B}),{\rm vech}^\top(\bm{\Lambda})\right)^\top$. If we conduct a log transformation to the response, then we will have a new variable $\widetilde{\bm{y}}_i=({\rm log} y_{i1},{\rm log} y_{i2})^\top$. After the transformation, the domain of the response $\widetilde{\bm{y}}_i$ is $\mathbb{R}^2$ and the log-transformed truncated Gaussian regression model is given by
\[
p(\widetilde{\bm{y}}_i|\bm{x}_i,\bm{\theta})=\frac{1}{Z(\bm{\theta})}\exp \left\{ \widetilde{\bm{y}}_i^\top \bm{l}-\frac{1}{2}\left(\bm{t}(\widetilde{\bm{y}}_i)-\bm{B}\bm{x}_i \right)^\top\bm{\Lambda} \left(\bm{t}(\widetilde{\bm{y}}_i)-\bm{B}\bm{x}_i  \right) \right\},
\]
where $\bm{l}=(1,1)^\top$ and $\bm{t}(\widetilde{\bm{y}}_i)=(e^{\widetilde{y}_{i1}},e^{\widetilde{y}_{i2}})^\top$. The score matching objective function is then in the form of
\begin{align}\label{log transformed tgm objective function}
\hat{d}_{\rm SM}(\bm{\theta})=&\frac{1}{n}\sum_{i=1}^n\Bigg\{ -4 \left(\bm{t}(\widetilde{\bm{y}}_i)-\bm{B}\bm{x}_i  \right)^\top\bm{\Lambda} \bm{t}(\widetilde{\bm{y}}_i)-2{\rm tr}\left(\bm{T}(\widetilde{\bm{y}}_i)\bm{\Lambda} \bm{T}(\widetilde{\bm{y}}_i)\right) \nonumber \\
&+\left(\bm{t}(\widetilde{\bm{y}}_i)-\bm{B}\bm{x}_i  \right)^\top \bm{\Lambda} \bm{T}(\widetilde{\bm{y}}_i) \bm{T}(\widetilde{\bm{y}}_i) \bm{\Lambda} \left(\bm{t}(\widetilde{\bm{y}}_i)-\bm{B}\bm{x}_i \right) \Bigg\},
\end{align}
where $\bm{T}(\widetilde{\bm{y}}_i)={\rm diag}(e^{\widetilde{y}_{i1}},e^{\widetilde{y}_{i2}})$. In Section \ref{sup sec::equi proof} of the supplementary material, we have shown that the objective function (\ref{log transformed tgm objective function}) of the log-transformed truncated Gaussian regression model is equivalent to the objective function of the truncated Gaussian regression model using weight functions (\citealt{yu2019generalized}).

For the truncated Gaussian regression model, Table \ref{tnormal:sim2} below reports the BIAS(SM), SD, ASD and RSME(SM) of the score matching estimator via $1,000$ replications with three sample sizes. According to Table \ref{tnormal:sim2}, we find that the absolute values of BIAS(SM) and SD generally become smaller for all parameter estimates as $n$ becomes larger. It is not surprising that RSME(SM) shows the same pattern. Furthermore, we notice that the absolute values of the difference between SD and ASD also become smaller for all estimators when $n$ gets larger. The above findings support our theoretical results that the score matching estimator for continuous data is consistent and asymptotically normal.

\begin{table}[H]
\caption{Comparison of the score matching estimation of the parameters ($B_{11}=1$, $B_{21}=0.4$, $B_{12}=-0.5$, $B_{22}=0.2$, $\Lambda_{11}=20$, $\Lambda_{12}=10$, $\Lambda_{22}=30$) for the truncated Gaussian regression model. Four measures are considered: the averaged bias of the estimate (BIAS(SM)), the standard deviation of the estimate (SD), the theoretical standard deviation of the estimate in the estimation theory (ASD), and the root mean squared error of the estimate (RMSE(SM)).\label{tnormal:sim2}}
\centering
\begin{tabular}{lclrrrrrrr}
  \hline
&$n$&Measure& $\hat{B}_{11}$ & $\hat{B}_{21}$ & $\hat{B}_{12}$ & $\hat{B}_{22}$ & $\hat{\Lambda}_{11}$ & $\hat{\Lambda}_{12}$ &  $\hat{\Lambda}_{22}$ \\ 
  \hline
&$n=200$&BIAS(SM) & 0.0015 & -0.0010 & 0.0006 & -0.0007 & 1.0372 & 0.5327 & 1.7498 \\ 
    &&SD   & 0.0313 & 0.0325 & 0.0367 & 0.0329 & 2.9603 & 2.4360 & 5.5817 \\   
     &&ASD& 0.0281 & 0.0255 & 0.0322 & 0.0270 & 2.6505 & 3.2547 & 4.4425 \\ 
    &&RMSE(SM)  & 0.0314 & 0.0325 & 0.0367 & 0.0329 & 3.1368 & 2.4936 & 5.8495 \\
   \hline
  &$n=500$&BIAS(SM)& 0.0010 & -0.0006 & 0.0005 & -0.0001 & 0.4147 & 0.0982 & 0.8257 \\ 
     &&SD  & 0.0188 & 0.0189 & 0.0240 & 0.0221 & 1.9426 & 1.5699 & 3.4566 \\ 
     &&ASD   & 0.0177 & 0.0160 & 0.0216 & 0.0183 & 1.7331 & 2.0596 & 2.8432 \\  
    &&RMSE(SM) & 0.0177 & 0.0160 & 0.0216 & 0.0183 & 1.7331 & 2.6596 & 2.8432 \\  
   \hline
   &$n=1000$&BIAS(SM) & -0.0001 & -0.0002 & -0.0000 & -0.0002 & 0.2134 & 0.1025 & 0.3003 \\   
    &&SD  & 0.0139 & 0.0148 & 0.0180 & 0.0154 & 1.2424 & 1.1253 & 2.4593 \\  
    &&ASD& 0.0131 & 0.0125 & 0.0156 & 0.0128 & 1.2441 & 1.5766 & 2.0818 \\  
    &&RMSE(SM)  & 0.0139 & 0.0148 & 0.0180 & 0.0154 & 1.2606 & 1.1299 & 2.4776 \\
   \hline
\end{tabular}
\end{table}

\subsection{Setting 2: simulation for CMP regression model}
For $i=1,\cdots,n$, we considered the $6\times 1$ covariate vector $\bm{x}_i$ which is randomly sampled from the original publication dataset in Section \ref{sect::case study}, and their corresponding regression parameters are the generalized score matching estimation, that is $\bm{\beta}_0=(\beta_1,\beta_2,\beta_3,\beta_4,\beta_5,\beta_6)^\top=(-0.3141,-0.0893,0.0445,-0.0705,0.0693,0.0830)^\top$. The true dispersion parameter is set to be $\nu_0=0.2564$ and the covariate matrix is fixed across the replications. Our parameter of interest is $\bm{\theta}=\left(\bm{\beta}^\top,\nu\right)^\top$ and the empirical objective function of the CMP regression model is given by
\[
\hat{d}_{\rm GSM}(\bm{\theta})=\frac{1}{n}\sum_{i=1}^n \rho_{\rm GSM}(y_i|\bm{x}_i,\bm{\theta}),
\]
where 
\[
\rho_{\rm GSM}(y_i|\bm{x}_i,\bm{\theta})=t^2\left(\frac{\lambda_i}{(y_i+1)^\nu}\right)+t^2\left(\frac{\lambda_i}{y_i^\nu}\right)-2t\left(\frac{\lambda_i}{(y_i+1)^\nu}\right),
\]
with $\lambda_i=\exp (\bm{x}_i^\top\bm{\beta})$. 

For the CMP regression model, Table \ref{normal:sim4} reports the BIAS(SM), SD, ASD and RSME(SM) of the generalized score matching estimator, along with the BIAS(AMLE) and RSME(AMLE) of the approximate MLE, via $1,000$ replications with three sample sizes. Table \ref{normal:sim4} indicates that the resulting estimation yields qualitatively similar conclusions to those obtained from the truncated Gaussian regression setting. Moreover, we notice that the generalized score matching method is much more accurate when the approximate MLE is biased. We further investigated the empirical coverages of a $95\%$ confidence interval constructed by generalized score matching estimation and approximate MLE and their asymptotic normal distributions, respectively, whose results are given in Table \ref{cmp coverage}. The above indicates that the inference of approximate MLE is not reliable. The above findings support our theoretical results that the generalized score matching estimation is consistent and asymptotically normal. As the original maximum likelihood estimation approach is intractable and the approximate MLE is biased under some situations, generalized score matching can be used to produce consistent estimation and reliable inference.

\begin{table}[H]
\caption{Comparison of the generalized score matching estimations and approximate MLEs of the parameters ($\beta_1=-0.3141$, $\beta_2=-0.0893$, $\beta_3=0.0445$, $\beta_4=-0.0705$, $\beta_5=0.0693$, $\beta_6=0.0830$ $\nu=0.2564$) for the CMP regression model. Six measures are considered: the averaged bias of the estimate (BIAS(SM)), the compared average bias of the estimate calculated by \code{glm.cmp} (BIAS(AMLE)), the standard deviation of the estimate (SD), the theoretical standard deviation of the estimate in the estimation theory (ASD), the root mean squared error of the estimate (RMSE(SM)), and the compared root mean squared error of the estimate calculated by \code{glm.cmp} (RMSE(AMLE)).\label{normal:sim4}}
\centering
\scalebox{0.9}{
\begin{tabular}{lclrrrrrrr}
  \hline
&$n$&Measure& $\hat{\beta}_1$ & $\hat{\beta}_2$ & $\hat{\beta}_3$ & $\hat{\beta}_4$ & $\hat{\beta}_5$ & $\hat{\beta}_6$ & $\hat{\nu}$ \\ 
  \hline
&$n=200$&BIAS(SM)& 0.0204 & 0.0053  & 0.0084 & -0.0051 & 0.0063& 0.0086 & 0.0643 \\
  &&BIAS(AMLE)  & 0.1986 & 0.0713 & 0.0211 & 0.0154 & 0.0292  & 0.0424 & 0.2950 \\
  &&SD  & 0.1970 & 0.1531 & 0.1598 & 0.0914 & 0.0743 & 0.0752 & 0.1491 \\
  &&ASD & 0.1871 & 0.1391 & 0.1466 & 0.0838 & 0.0700 & 0.0677 & 0.1528 \\
  &&RMSE(SM)& 0.1981 & 0.1532 & 0.1600 & 0.0915 & 0.0746 & 0.0756 & 0.1623 \\
  &&RMSE(AMLE) & 0.3005 & 0.1945 & 0.1831 & 0.0917 & 0.0817 & 0.0845 & 0.3175 \\
   \hline
  &$n=500$&BIAS(SM)& 0.0152 & -0.0050 & -0.0038 &  -0.0034 & 0.0060 & 0.0037 & 0.0391 \\
&&BIAS(AMLE)  & 0.2076 & 0.0517 & -0.0198 & 0.0123 & 0.0328 & 0.0227 & 0.2825 \\
 &&SD & 0.1248  & 0.0996 & 0.1055 & 0.0515 & 0.0466 & 0.0385 & 0.1047 \\
 &&ASD & 0.1170 & 0.0850 & 0.0944 & 0.0481 & 0.0419 & 0.0337 & 0.0963 \\ 
&&RMSE(SM) & 0.1257 & 0.0997 & 0.1055 & 0.0516& 0.0470 & 0.0387 & 0.1118 \\
&&RMSE(AMLE) & 0.2585 & 0.1432 & 0.1420 & 0.0659 & 0.0656 & 0.0519 & 0.2992 \\
   \hline
 &$n=1000$&BIAS(SM) & 0.0132 & 0.0032 & -0.0017 & -0.0019 & 0.0011 & 0.0013 & 0.0254 \\
&&BIAS(AMLE)  & 0.2153 & 0.0533 & -0.0061 & 0.0130 & 0.0347 & 0.0336& 0.2899 \\
 &&SD  & 0.0889 & 0.0695 & 0.0770 & 0.0392 & 0.0346 & 0.0274 & 0.0798 \\
&&ASD & 0.0819 & 0.0606 & 0.0669 & 0.0357 & 0.0302 & 0.0240 & 0.0746 \\  
&&RMSE(SM) & 0.0899 & 0.0695 & 0.0770 & 0.0393 & 0.0346 & 0.0274 & 0.0837 \\  
&&RMSE(AMLE) & 0.2563 & 0.1353 & 0.1266 & 0.0602 & 0.0617 & 0.0507 & 0.3041 \\
   \hline
\end{tabular}}
\end{table}
\noindent

\begin{table}[H]
\caption{Comparison of the generalized score matching estimations and approximate MLEs of the parameters ($\beta_1=-0.3141$, $\beta_2=-0.0893$, $\beta_3=0.0445$, $\beta_4=-0.0705$, $\beta_5=0.0693$, $\beta_6=-1.3610$ $\nu=0.2564$) for the CMP regression model. One measure is considered: the empirical coverages of a $95\%$ confidence interval constructed by generalized score matching estimations and approximate MLEs and their asymptotic normal distributions, respectively.\label{cmp coverage}}
\centering
\begin{tabular}{lcclrrrr}
  \hline
&Estimation&$n$&Measure& $\beta_1$ & $\beta_2$ & $\beta_3$ & $\beta_4$  \\ 
  \hline
&\multirow{3}{*}{Score matching}&$n=200$&Coverage& 0.9280 &0.9460 &0.9260 &0.9460 \\
  &&$n=500$ &Coverage & 0.9420 &0.9440& 0.9360& 0.9540   \\
  &&$n=1000$ &Coverage& 0.9580 &0.9480 &0.9540& 0.9580 \\
  \hline
&\multirow{3}{*}{Approximate MLE}&$n=200$&Coverage& 0.7600& 0.7900& 0.8060& 0.8820 \\
  &&$n=500$ &Coverage & 0.4820& 0.6160& 0.7240 &0.7460   \\
  &&$n=1000$ &Coverage& 0.3980 &0.4120& 0.6280 &0.6860  \\
  \hline
&Estimation&$n$&Measure& $\beta_5$ & $\beta_6$ & $\nu$  \\ 
  \hline
  &\multirow{3}{*}{Score matching}&$n=200$&Coverage& 0.9580 &0.9580&0.9600 \\
  &&$n=500$  &Coverage & 0.9420 &0.9460 &0.9460   \\
  &&$n=1000$ &Coverage & 0.9520 &0.9520 &0.9580\\
   \hline
&\multirow{3}{*}{Approximate MLE}&$n=200$&Coverage& 0.8500 &0.8260 &0.8200 \\
  &&$n=500$  &Coverage &0.7140 &0.7100 &0.4320  \\
  &&$n=1000$ &Coverage &0.6260 &0.5380&0.2220\\
   \hline
\end{tabular}
\end{table}

\section{Conclusion}

In this article, we propose generalized score matching methods for probabilistic regression-type models with an intractable normalizing constant and derive the consistency and asymptotic normality of the estimators. Specifically, we consider a novel score matching estimator for continuous multivariate data with covariates and present a novel generalized score matching approach for count data. The proposed generalized score matching approach goes beyond previous research in that it can be applied to univariate and multivariate count data.

The simulations and real data analysis support our theoretical results. In addition, we see that the generalized score matching approach can provide accurate estimation and reliable inference for CMP regression models in situations where the approximate MLE yields biased estimates.

\section*{Appendix}

\begin{proof}[Proof of Theorem \ref{thm:: tractable obj funct for GSM}]
Recall that the transformation function $t(f_1(y)/f_2(y))$ will give a zero value when $f_2(y)=0$ and $t(f_1(y)/f_2(y))=1$ when $f_1(y)=0$. In other words, the transformation $t(f_1(y)/f_2(y))$ will return a constant when either $f_1(y)$ or $f_2(y)$ equals zero. Therefore, for the sake of simplicity, we assume that $q(y|\bm{x}_i)$ and $p(y|\bm{x}_i,\bm{\theta})$ are non-zero for $y\in\mathcal{D}$. Note that we have
\begin{align*}
D_{\rm GSM}(q_i,p_i)=&\sum_{y\in\mathcal{D}}q(y|\bm{x}_i) \Bigg\{ t\left( \frac{p(y^+|\bm{x}_i,\bm{\theta})}{p(y|\bm{x}_i,\bm{\theta})} \right)^2 + t\left( \frac{p(y|\bm{x}_i,\bm{\theta})}{p(y^-|\bm{x}_i,\bm{\theta})}\right)^2 \Bigg\}\\
&-2\sum_{y\in\mathcal{D}}q(y|\bm{x}_i) \Bigg\{ t\left( \frac{p(y^+|\bm{x}_i,\bm{\theta})}{p(y|\bm{x}_i,\bm{\theta})} \right) t\left( \frac{q(y^+)}{q(y)} \right) \\
&  + t\left( \frac{p(y|\bm{x}_i,\bm{\theta})}{p(y^-|\bm{x}_i,\bm{\theta})}\right) t\left( \frac{q(y|\bm{x}_i)}{q(y^-|\bm{x}_i)} \right)  \Bigg\}+C,
\end{align*}
where $C$ does not depend on $\bm{\theta}$. We first consider the second term, which can be manipulated as follows:
\begin{align*}
&\sum_{y\in\mathcal{D}}q(y|\bm{x}_i) \Bigg\{ t\left( \frac{p(y^+|\bm{x}_i,\bm{\theta})}{p(y|\bm{x}_i,\bm{\theta})} \right) t\left( \frac{q(y^+|\bm{x}_i)}{q(y|\bm{x}_i)} \right)+ t\left( \frac{p(y|\bm{x}_i,\bm{\theta})}{p(y^-|\bm{x}_i,\bm{\theta})}\right) t\left( \frac{q(y|\bm{x}_i)}{q(y^-|\bm{x}_i)} \right)  \Bigg\}\nonumber\\
=&\sum_{y\in\mathcal{D}}q(y|\bm{x}_i) \Bigg\{ t\left( \frac{p(y^+|\bm{x}_i,\bm{\theta})}{p(y|\bm{x}_i,\bm{\theta})} \right)  \frac{q(y|\bm{x}_i)}{q(y^+|\bm{x}_i)+q(y|\bm{x}_i)} +t\left( \frac{p(y|\bm{x}_i,\bm{\theta})}{p(y^-|\bm{x}_i,\bm{\theta})}\right) \frac{q(y^-|\bm{x}_i)}{q(y^-|\bm{x}_i) +q(y|\bm{x}_i)}  \Bigg\}\\
=&\sum_{y\in\mathcal{D}}q(y|\bm{x}_i) t\left( \frac{p(y^+|\bm{x}_i,\bm{\theta})}{p(y|\bm{x}_i,\bm{\theta})} \right)  \frac{q(y|\bm{x}_i)}{q(y^+|\bm{x}_i)+q(y|\bm{x}_i)} \\
&+ \sum_{y\in\mathcal{D}}q(y^+|\bm{x}_i) t\left( \frac{p(y^+|\bm{x}_i,\bm{\theta})}{p(y|\bm{x}_i,\bm{\theta})}\right) \frac{q(y|\bm{x}_i)}{q(y^+|\bm{x}_i) +q(y|\bm{x}_i)}\\
=& \sum_{y\in\mathcal{D}}q(y|\bm{x}_i) t\left( \frac{p(y^+|\bm{x}_i,\bm{\theta})}{p(y|\bm{x}_i,\bm{\theta})} \right).
\end{align*}

Therefore, we can conclude that
\begin{align*}
D_{\rm GSM}(q_i,p_i)=&\sum_{y\in\mathcal{D}}q(y|\bm{x}_i)\Bigg\{ t\left( \frac{p(y^+|\bm{x}_i,\bm{\theta})}{p(y|\bm{x}_i,\bm{\theta})} \right)^2\\
&+ t\left( \frac{p(y|\bm{x}_i,\bm{\theta})}{p(y^-|\bm{x}_i,\bm{\theta})}\right)^2-2t\left( \frac{p(y^+|\bm{x}_i,\bm{\theta})}{p(y|\bm{x}_i,\bm{\theta})} \right) \Bigg\}+C,
\end{align*}
which completes the entire proof.

\end{proof}

\begin{proof}[Proof of Theorem \ref{thm:: GSM unique solution}]
Based on the analysis of the transformation function $t(\cdot)$ in the proof of Theorem \ref{thm:: tractable obj funct for GSM}, we assume that $q(y|\bm{x}_i)$ and $p(y|\bm{x}_i,\bm{\theta})$ are non-zero for $y\in\mathcal{D}$. The hypothesis $D_{\rm GSM}(q_*,p_*)=0$, in conjunction with the assumption that $p(y|\bm{x}_i,\bm{\theta}_0)=q(y|\bm{x}_i)>0$, implies that all the slopes must be equal for the model and the observed data. Thus, we have
\begin{align}
\frac{p(y^+|\bm{x}_i,\bm{\theta})}{p(y|\bm{x}_i,\bm{\theta})}&=\frac{q(y^+|\bm{x}_i)}{q(y|\bm{x}_i)},\label{forward}\\
\frac{p(y|\bm{x}_i,\bm{\theta})}{p(y^-|\bm{x}_i,\bm{\theta})}&=\frac{q(y|\bm{x}_i)}{q(y^-|\bm{x}_i)} ,\label{backward}
\end{align} 
for all $y$ and $i=1,\cdots,n$. It is worth noting that the relationships (\ref{forward}) and (\ref{backward}) are equivalent. Without loss of generality, we only consider the first case (\ref{forward}) with $y^+\in\mathcal{D}$. We then obtain
\[
\frac{q(y|\bm{x}_i) }{p(y|\bm{x}_i,\bm{\theta})}=\frac{q(y^+|\bm{x}_i)}{p(y^+|\bm{x}_i,\bm{\theta})}.
\]
Applying this identity on $y_i^+$, we get
\[
\frac{q(y^+|\bm{x}_i)}{p(y^+|\bm{x}_i,\bm{\theta})}=\frac{q(y^{++}|\bm{x}_i)}{p(y^{++}|\bm{x}_i,\bm{\theta})},
\]
for $y^{++}=y+2$ and $y^{++}\in\mathcal{D}$. It can be readily seen that we can recursively apply this identity.

Now, fix any point $y^0$. Without loss of generality, we assume that $y^0\geq y$. By using the recursion above, we have
\[
\frac{q(y|\bm{x}_i) }{p(y|\bm{x}_i,\bm{\theta})}=\frac{q(y^+|\bm{x}_i)}{p(y^+|\bm{x}_i,\bm{\theta})}=\cdots=\frac{q(y^0|\bm{x}_i) }{p(y^0|\bm{x}_i,\bm{\theta})}=c,
\]
where $c$ is a constant does not depend on $y$. Therefore, we can conclude that
\[
q(y|\bm{x}_i)=c p(y|\bm{x}_i,\bm{\theta})
\]
for any $y$.

On the other hand, both $p_i$ and $q_i$ are normalized probability distributions. Thus, we must have $c=1$. This proves that if $D_{\rm GSM}(q_*,p_*)=0$, then $q(y|\bm{x}_i)=p(y|\bm{x}_i,\bm{\theta})$ for any $y$ and $i=1,\cdots,n$. Using the identifiability assumption, this implies $\bm{\theta}=\bm{\theta}_0$. Thus, we have proved that $D_{\rm GSM}(q_*,p_*)=0$ implies $\bm{\theta}=\bm{\theta}_0$. The converse is trivial.
\end{proof}

\begin{proof}[Proof of Proposition \ref{thm:: sm score function expectation}]
This proof consists of two parts. First, we consider the case when $\hat{d}=\hat{d}_{\rm SM}$ and then we derive results for $\hat{d}=\hat{d}_{\rm GSM}$.

\emph{Part I}: When $\hat{d}=\hat{d}_{\rm SM}$, $\hat{d}_{\rm SM}(\bm{\theta})=\frac{1}{n}\sum_{i=1}^n \rho_{\rm SM}(\bm{y}_i|\bm{x}_i,\bm{\theta})$, we have
\begin{align}\label{exp eq 1}
\mathbb{E}\left(\frac{\partial \hat{d}_{\rm SM}(\bm{\theta})}{\partial \bm{\theta}} \right)&=\frac{1}{n}\sum_{i=1}^n \mathbb{E}\left(\frac{\partial \rho_{\rm SM}(\bm{y}_i|\bm{x}_i,\bm{\theta})}{\partial \bm{\theta}} \right) =\frac{1}{n}\sum_{i=1}^n\frac{\partial}{\partial\bm{\theta}}\mathbb{E}\left[ \rho_{\rm SM}(\bm{y}|\bm{x}_i,\bm{\theta}) \right].
\end{align}
The last equality holds based on the interchange of integration and differentiation under Condition (C1). After interchanging the differentiation and integral sign, (\ref{exp eq 1}) then implies that
\begin{align*}
\mathbb{E}\left(\frac{\partial \hat{d}_{\rm SM}(\bm{\theta})}{\partial \bm{\theta}} \right)&=\frac{1}{n}\sum_{i=1}^n \frac{\partial}{\partial\bm{\theta}} \mathbb{E}\left( \|\nabla_{\bm{y}}\log q(\bm{y}_i|\bm{x}_i)-\nabla_{\bm{y}}\log p(\bm{y}_i|\bm{x}_i,\bm{\theta})\|^2 \right)\\
&= \frac{1}{n}\sum_{i=1}^n  \mathbb{E}\left(  \frac{\partial}{\partial\bm{\theta}} \|\nabla_{\bm{y}}\log q(\bm{y}_i|\bm{x}_i)-\nabla_{\bm{y}}\log p(\bm{y}_i|\bm{x}_i,\bm{\theta})\|^2 \right)\\
&= -\frac{1}{n}\sum_{i=1}^n  \mathbb{E}\left( 2\nabla_{\bm{y}}\left(\log q(\bm{y}_i|\bm{x}_i)-\log p(\bm{y}_i|\bm{x}_i,\bm{\theta})\right)  \frac{\partial}{\partial\bm{\theta}} \nabla_{\bm{y}}\log p(\bm{y}_i|\bm{x}_i,\bm{\theta}) \right).
\end{align*}
It can be readily seen that $\mathbb{E}\left(\frac{\partial \hat{d}_{\rm SM}(\bm{\theta}_0)}{\partial \bm{\theta}}\right)=\bm{0}$ since $\log q(\bm{y}|\bm{x}_i)=\log p(\bm{y}|\bm{x}_i,\bm{\theta}_0)$. This completes the first part of proof.

\emph{Part II}: When $\hat{d}=\hat{d}_{\rm GSM}$, under Condition (C1), we have
\begin{align*}
\mathbb{E}\left( \frac{\partial \hat{d}_{\rm GSM}(\bm{\theta})}{\partial \bm{\theta}} \right) =&\frac{1}{n}\sum_{i=1}^n \frac{\partial}{\partial \bm{\theta}}\mathbb{E}\left[ \rho_{\rm GSM}(y_i|\bm{x}_i,\bm{\theta})\right]\\
=& \frac{1}{n}\sum_{i=1}^n \frac{\partial}{\partial \bm{\theta}}\mathbb{E}\left[ D_{\rm GSM}(q_i,p_i)\right]\\
=& \frac{1}{n}\sum_{i=1}^n \sum_{y\in\mathcal{D}} q(y|\bm{x}_i) \frac{\partial}{\partial \bm{\theta}}  \Bigg\{\Bigg[ t\left( \frac{p(y^+|\bm{x}_i,\bm{\theta})}{p(y|\bm{x}_i,\bm{\theta})} \right) - t\left( \frac{q(y^+|\bm{x}_i)}{q(y|\bm{x}_i)} \right) \Bigg]^2\\
& + \Bigg[ t\left( \frac{p(y|\bm{x}_i,\bm{\theta})}{p(y^-|\bm{x}_i,\bm{\theta})}\right) - t\left( \frac{q(y|\bm{x}_i)}{q(y^-|\bm{x}_i)} \right) \Bigg]^2 \Bigg\}\\
=&\frac{1}{n}\sum_{i=1}^n \sum_{y\in\mathcal{D}} q(y|\bm{x}_i)  \Bigg\{ 2\Bigg[ t\left( \frac{p(y^+|\bm{x}_i,\bm{\theta})}{p(y|\bm{x}_i,\bm{\theta})} \right) - t\left( \frac{q(y^+|\bm{x}_i)}{q(y|\bm{x}_i)} \right) \Bigg] \\
&\times\frac{\partial}{\partial \bm{\theta}} \Bigg[ t\left( \frac{p(y^+|\bm{x}_i,\bm{\theta})}{p(y|\bm{x}_i,\bm{\theta})} \right) - t\left( \frac{q(y^+|\bm{x}_i)}{q(y|\bm{x}_i)} \right) \Bigg]\\
& + 2\Bigg[ t\left( \frac{p(y|\bm{x}_i,\bm{\theta})}{p(y^-|\bm{x}_i,\bm{\theta})}\right) - t\left( \frac{q(y|\bm{x}_i)}{q(y^-|\bm{x}_i)} \right) \Bigg]\\
&\times \frac{\partial}{\partial \bm{\theta}} \Bigg[ t\left( \frac{p(y|\bm{x}_i,\bm{\theta})}{p(y^-|\bm{x}_i,\bm{\theta})}\right) - t\left( \frac{q(y|\bm{x}_i)}{q(y^-|\bm{x}_i)} \right) \Bigg] \Bigg\}.
\end{align*}
One can be readily seen that 
\[
\mathbb{E}\left(\frac{\partial \hat{d}_{\rm GSM}(\bm{\theta}_0)}{\partial \bm{\theta}}\right)=\bm{0},
\]
which completes the entire proof.
\end{proof}

\begin{proof}[Proof of Proposition \ref{prop::clt score function}]
We use the Lyapounov theorem for triangular arrays to derive our results. For the proof of the first part of Proposition \ref{prop::clt score function}, we first consider $\hat{d}=\hat{d}_{\rm SM}$ and define a triangular array
\[
\bm{z}_{ni}=\frac{\partial \rho_{\rm SM}(\bm{y}_i|\bm{x}_i,\bm{\theta}_0)}{\partial \bm{\theta}},
\]
for $i=1,\cdots,n$. Note that we have $\mathbb{E}(\bm{z}_{ni})=\bm{0}$ and $\bm{z}_{ni}$ is actually a function of $\bm{y}_i$. Thus  $\bm{z}_{n1},\cdots,\bm{z}_{nn}$ are INID random variables with mean $\bm{0}$. Note that the Condition (C2) gives
\[
\frac{1}{n}{\rm Cov}\left(\sum_{i=1}^n\bm{z}_{ni}\right)=\mathcal{I}_n(\bm{\theta}_0)+\mathcal{J}_n(\bm{\theta}_0)\to\mathcal{I}(\bm{\theta}_0)+\mathcal{J}(\bm{\theta}_0),
\]
by Lyapounov theorem and the Lyapounov condition (C3), we get
\[
n^{\frac{1}{2}}\frac{\partial \hat{d}_{\rm SM}(\bm{\theta}_0)}{\partial \bm{\theta}}=n^{-\frac{1}{2}}\sum_{i=1}^n\bm{z}_{ni}\stackrel{d}{\longrightarrow}N\left(\bm{0},\mathcal{I}(\bm{\theta}_0)+\mathcal{J}(\bm{\theta}_0)\right).
\]

For the proof of the second part of Proposition \ref{prop::clt score function}, we slightly abuse the notion of $\bm{z}_{ni}$ and use this to define a new triangle array
\[
\bm{z}_{ni}={\rm vec}\left(\frac{\partial^2 \rho_{\rm SM}(\bm{y}_i|\bm{x}_i,\bm{\theta}_0)}{\partial \bm{\theta} \partial \bm{\theta}^\top}\right).
\] 
Under Condition (C4), one can be readily seen that
\begin{align*}
\sup_n \max_i \mathbb{E}\|\bm{z}_{ni}\|^2= \sup_n\max_i \mathbb{E}\left(\left\|\frac{\partial^2 \rho_{\rm SM}(\bm{y}_i|\bm{x}_i,\bm{\theta}_0)}{\partial \bm{\theta} \partial \bm{\theta}^\top}\right\|_F^2 \right)
<\infty,
\end{align*}
thus by the weak law of large numbers and Condition (C2), we obtain
\[
-\frac{\partial^2 \hat{d}_{\rm SM}(\bm{\theta}_0)}{\partial \bm{\theta} \partial \bm{\theta}^\top}\stackrel{p}{\longrightarrow}\mathcal{I}(\bm{\theta}_0),
\]
which completes the proof for $\hat{d}=\hat{d}_{\rm SM}$. After employing similar methods to those used in the proof for $\hat{d}=\hat{d}_{\rm SM}$, we could finish the proof for $\hat{d}=\hat{d}_{\rm GSM}$.
\end{proof}

\bigskip
\begin{center}
{\large\bf SUPPLEMENTARY MATERIAL}
\end{center}
The supplementary material consists of the equivalence of the score matching loss functions for truncated Gaussian regression models, generalized score matching for discrete data with multivariate response and derivatives of proposed score matching objective functions.

\bibliographystyle{apalike}
 \bibliography{Reference}

\end{document}